\documentclass[11pt]{article}
\usepackage{amsthm}
\usepackage{latexsym,amssymb,amsmath}
\usepackage{graphicx,float,color,fancybox,shapepar,setspace,hyperref}
\usepackage{subfigure}
\usepackage{mathrsfs}
\usepackage{cite}
\usepackage{tabularx}
\usepackage{pgf,tikz}
\usetikzlibrary{arrows}
\usetikzlibrary{shapes.geometric}
\newtheorem{thm}{Theorem}[section]
\newtheorem{case}[]{Case}
\newtheorem{subcase}[]{Subcase}[case]

\newtheorem{cor}[thm]{Corollary}%[section]
\newtheorem{lem}[thm]{Lemma}%[section]
\newtheorem*{obs}{Observation}
\newtheorem*{que}{Question}
\newtheorem{claim}[]{Claim}

\def\qed{\hfill\square}

\def\qed{ \hfill $\blacksquare$}

\begin{document}
\title{Bonds intersecting long paths in $k$-connected graphs}

\author{Bing Wei\thanks{Supported in part by the summer faculty research grant from CLA at the University of Mississippi.},   ~Haidong Wu, Qinghong Zhao\thanks{Corresponding Author. Supported by the Summer Graduate Research Assistantship Program of Graduate School at the University of Mississippi. E-mail address: qzhao1@olemiss.edu (Q. Zhao)}\\ \small Department of Mathematics, University of Mississippi, University,  MS 38677,  USA}

\date{}
\maketitle	
\begin{abstract}
A well-known question of Gallai (1966) \cite{G} asked whether there is a vertex which passes through all longest paths of a connected graph. Although this has been verified for some special classes of graphs such as outerplanar graphs, circular arc graphs, and series-parallel graphs \cite{B1,J, R, C2}, the answer is negative for general graphs.  In this paper, we prove among other results that if we replace the vertex by a bond, then the answer is affirmative. A {\it bond} of a graph is a minimal nonempty edge-cut. In particular, in any 2-connected graph, the set of all edges incident to a vertex is a bond, called a vertex-bond. Clearly, for a 2-connected graph, a path passes through a vertex $v$ if and only if it meets the vertex-bond with respect to $v$. Therefore, a very natural approach to Gallai's question is to study whether there is a bond meeting all longest paths. Let $p$ denote the length of a longest path of connected graphs. We show that  for any 2-connected graph, there is a bond meeting all paths of length at least $p-1$. We then prove that for any 3-connected graph,  there is a bond meeting all paths of length at least $p-2$. For a $k$-connected graph $(k\ge3)$, we show that there is a bond meeting all paths of length at least $p-t+1$, where $t=\Big\lfloor\sqrt{\frac{k-2}{2}}\Big\rfloor$ if $p$ is even and $t=\Big\lceil\sqrt{\frac{k-2}{2}}\Big\rceil$ if $p$ is odd. Our results provide analogs of the corresponding results of P. Wu \cite{M1} and S. McGuinness \cite[Bonds intersecting cycles in a graph, Combinatorica 25 (4) (2005), 439-450]{W3} also. 

\end{abstract}

\section{Introduction}
The intersection of longest paths in a graph has interesting and somewhat surprising behaviors and it has been studied for a long time. In 1966, Gallai \cite{G} asked the following well-known question: 
\begin{que}
Does every connected graph have a vertex that is common to all of its longest paths?
\end{que}
\noindent

The answer to Gallai's question is affirmative for some special classes of graphs such as split graphs \cite{K}, circular-arc graphs \cite{B1,J}, outerplanar graphs and 2-trees \cite{R},  series-parallel graphs \cite{C2}, dually chordal graphs \cite{J1}, $2K_2$-free graphs \cite{G1} and more \cite{C3,W4}. We refer the readers to a survey \cite{S} on Gallai's question for more information. For general graphs, however, the answer to Gallai's question is negative.  The $hypotraceable$ graph, that is, a graph has no Hamilton path but all of whose vertex-deleted subgraphs have, is obviously a counterexample and Thomassen \cite{T} proved the existence of infinitely many such graphs. Moreover,  Walther and Voss \cite{W2} and Zamfirescu \cite{Z} answered this question negatively by exhibiting a counterexample on 12 vertices. Later, \cite{BC,C1} verified that it is a smallest counterexample.

It is natural to consider the question of replacing the vertex by a special set instead in Gallai's question. A bond of a graph is a minimal nonempty edge-cut. In particular, in any 2-connected graph, the set of all edges incident to a vertex is a bond, called a {\it vertex-bond} associated with $v$. Clearly, the following is true:

\begin{obs}
In a 2-connected graph, a path passes through a vertex $v$ if and only if it meets the vertex-bond associated with $v$.
\end{obs}

Therefore,  a natural approach to Gallai's question is to consider the intersection of bonds and longest paths. In 1997, Wu \cite{W3} proved the following result when he studied the upper bound on the number of edges of a 2-connected graph.
\begin{thm}[\cite{W3}]\label{W3}
Let $G$ be a 2-connected graph with longest cycle of length $c\ge4$. Then there is a bond meeting every cycle of length at least $c-1$.
\end{thm}

\noindent
In 2005, McGuinness proved  a corresponding result  for $k$-connected graph, as follows.

\begin{thm}[\cite{M1}]\label{M1}
Let $G$ be a $k$-connected graph with longest cycle of length $c$, where $c\ge2k$ and $k\ge2$. Then there is a bond meeting every cycle of length at least $c-k+2$.
\end{thm}

\noindent
Inspired by above observation and Theorems \ref{W3} and \ref{M1}, in this paper, we study the intersection of bonds and long paths of a connected graph and obtain the following results.
\begin{thm}\label{thm1}
Let $G$ be a connected graph of order $n\ge4$ with longest path of length $p$. Then one of the following statements holds:
\begin{enumerate}
\item There is a bond meeting all paths of length at least $p-1$;
\item there is a cut-edge (one-element bond) $uv$ meeting all longest paths, where $p$ is odd. Moreover, $\{u,v\}$ meets all paths of length at least $p-1$;
\item there is a cut-vertex passing through all paths of length at least $p-1$.
\end{enumerate}
\end{thm}

This result easily implies the following two corollaries.

\begin{cor}
Let $G$ be a 2-connected graph of order $n\ge4$ with longest path of length $p$. Then there is a bond meeting all paths of length at least $p-1$.
\end{cor}

\begin{cor}
Let $G$ be a tree with longest path of length $p$. Then either
\begin{itemize}
\item there is a cut-edge (one-element bond) $uv$ meeting all longest paths, where $p$ is odd. Moreover, $\{u,v\}$ meets all paths of length at least $p-1$; or
\item there is a cut-vertex passing through all paths of length at least $p-1$.
\end{itemize}
\end{cor}

For 3-connected graphs, we prove the following improvement of the last result. 

\begin{thm}\label{thm2}
Let $G$ be a 3-connected graph of order $n\ge6$ with longest path of length $p$. Then there is a bond meeting all paths of length at least $p-2$.	
\end{thm}

For $k$-connected graphs ($k\ge 3$), we prove that we can find a bond meeting all paths of length at least $p-t+1$ where $t=\Omega(\sqrt{k})$. 

\begin{thm}\label{thmx}
Let $G$ be a $k$-connected graph $(k\ge3)$ with longest path of length $p$. Then there is a bond meeting all paths of length at least $p-t+1$, where $t=\Big\lfloor\sqrt{\frac{k-2}{2}}\Big\rfloor$ if $p$ is even and $t=\Big\lceil\sqrt{\frac{k-2}{2}}\Big\rceil$ if $p$ is odd.
\end{thm}
Another related intriguing concept to Gallai's question is the notion of \textit{transversal} of a graph, which is defined to be the smallest set of vertices that intersects every longest path. Given a graph $G$, the size of transversal is denoted by $lpt(G)$ sometimes. Clearly, the answer to Gallai's question is affirmative when $lpt(G)=1$ for a connected graph $G$. So far, the upper bounds of $lpt(G)$ for many classes of connected graphs, especially to chordal graph, have been studied. In 2014, Rautenbach and Sereni \cite{R1} proved that $lpt(G)\le\lceil\frac{n}{4}-\frac{n^{2/3}}{90}\rceil$ for every connected graph $G$ on $n$ vertices, and $lpt(G)\le k+1$ for every connected partial $k$-tree $G$. The latter result implies that $lpt(H)\le\omega(H)$, where $H$ is a connected chordal graph and $\omega(H)$ is the clique number of $H$. Furthermore, Cerioli et al. \cite{C3} showed that $lpt(H)\le max\{1,\omega(H)-2\}$. Recently, Harvey and Payne \cite{H} improved this result and proved that $lpt(H)\le4\cdot\lceil\frac{\omega(H)}{5}\rceil$. Our theorem \ref{thm1} directly shows the following result. 
\begin{cor}\label{thmy}
Let $G$ be a connected graph. Then $lpt(G)\le max|C^*|$, where $C^*$ is a bond of $G$. 
\end{cor}

We organize our paper as follows. In Section 2, we first define some notations and then prove many lemmas which will be used later in our proofs. Finally, we complete the proofs of our main results in Section 3.

\section{Preliminaries}
We start this section with some notations. All graphs considered are finite, simple and undirected. Given a graph $G$, we denote by $V(G)$ the vertex set of $G$, by $|V(G)|$ the order of $G$, and by $\delta(G)$ the minimum degree of $G$. For a subset $S\subseteq V(G)$, let $G[S]$ denote the subgraph induced by the vertices of $S$, and we simply write $G-S$ as $G[V(G)-S]$. Given disjoint subsets $X,Y\subseteq V(G)$, an $(X,Y)$-path is defined to be a path which has its initial vertex in $X$, its terminal vertex in $Y$ and no internal vertex in $X\cup Y$. A path that contains every vertex of a graph is called a Hamilton path. Let $P=x_0x_1\ldots x_r$ be a path of $G$. We denote by $\ell(P)$ the length or the number of edges of $P$, and by $P[x_i,x_j]$ a subpath of $P$ from $x_i$ to $x_j$. 

Let $P_1=ux_1\ldots x_r$, $P_2=uy_1\ldots y_s$ and $P_3=uz_1\ldots z_t$ be three paths of a connected graph $G$ such that any two of them intersect only at the vertex $u$. Let $R_1=x_i\cdots y_j$ and $R_2=x_p\cdots y_q$ be two vertex-disjoint $(V(P_1),V(P_2))$-paths in $G-V(P_3)$. If $(i-p)(j-q)>0$, then we say $R_1$ and $R_2$ are parallel (see Fig. 1$(b)$). If $(i-p)(j-q)<0$, then we say $R_1$ crosses $R_2$ (see Fig. 1($c$)). We next prove several lemmas.
\begin{lem}\label{lem}
Let $G$ be a connected graph with longest path of length $p$. Let $P_1=ux_1\ldots x_r$, $P_2=uy_1\ldots y_s$ and $P_3=uz_1\ldots z_t$ be three paths of $G$ such that any two of them intersect only at the vertex $u$. If $2\ell(P_3)+\ell(P_1)+\ell(P_2)\ge2p-1$ $(resp.$, $2\ell(P_3)+\ell(P_1)+\ell(P_2)\ge 2p-3)$, then there is no $(V(P_1),V(P_2))$-path $(resp.$, two vertex-disjoint $(V(P_1),V(P_2))$-paths$)$ in $G-V(P_3)$.
\end{lem}
\begin{center}
	\includegraphics[scale=.55]{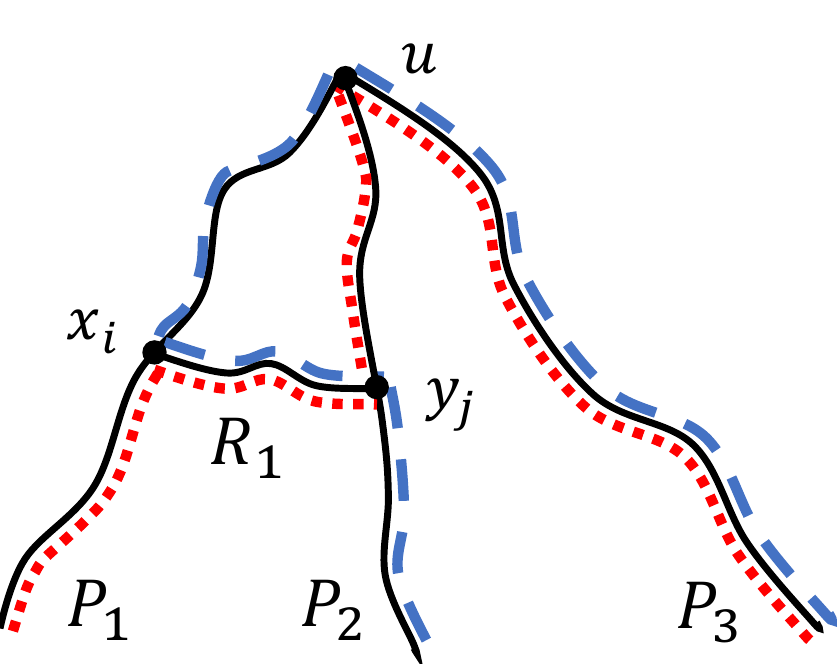}
	\hfil
	\includegraphics[scale=.55]{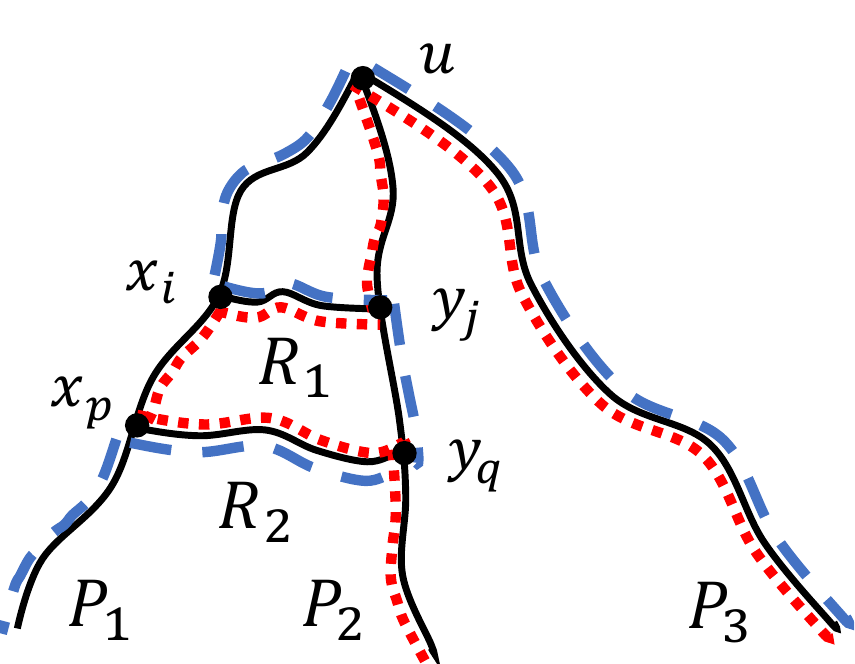}
	\hfil
	\includegraphics[scale=.55]{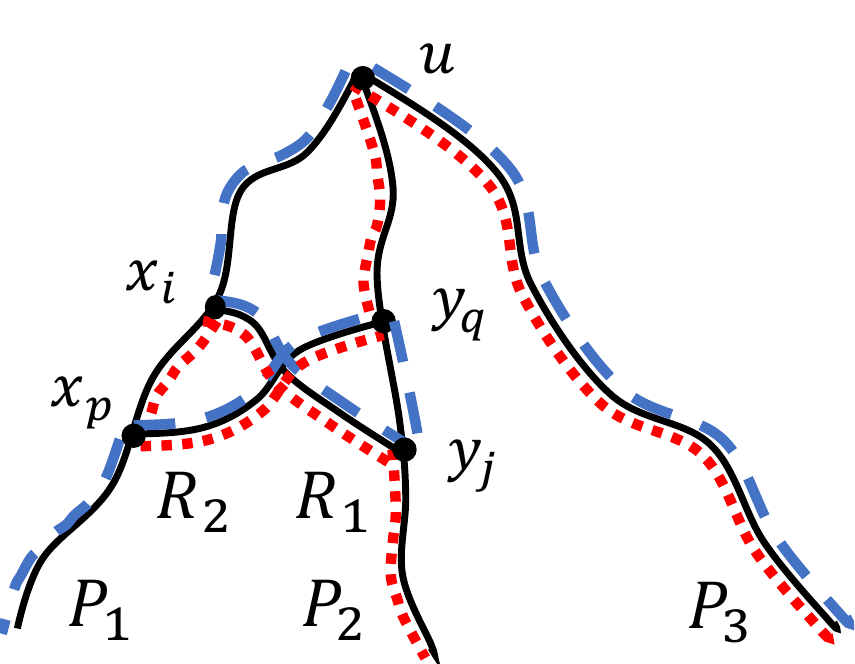}
	\hfil
($a$)~~~~~~~~~~~~~~~~~~~~~~~~~~~~~~~~~~~~~~($b$)~~~~~~~~~~~~~~~~~~~~~~~~~~~~~~~~~~~~~~($c$)
\end{center}
\begin{center}
Fig. 1. Configurations of two paths.
\end{center}
\begin{proof}
Suppose not. We may first assume that $2\ell(P_3)+\ell(P_1)+\ell(P_2)\ge2p-1$. Let $R_1=x_i\cdots y_j$ be a $(V(P_1),V(P_2))$-path in $G-V(P_3)$. Then there exist two paths $\widetilde{P}=P_3P_1[u,x_i]R_1P_2[y_j,y_m]$ and $\widetilde{Q}=P_3P_2[u,y_j]R_1P_1[x_i,x_r]$ in $G$ (see Fig. 1$(a)$) such that 
\begin{center}
$\ell(\widetilde{P})+\ell(\widetilde{Q})=2\ell(P_3)+\ell(P_1)+\ell(P_2)+2\ell(R_1)$
\end{center}
\begin{center}
$\ge2p-1+2>2p,~~~$
\end{center}
which is a contradiction. We next assume that $2\ell(P_3)+\ell(P_1)+\ell(P_2)\ge 2p-3$. Let $R_1=x_i\cdots y_j$ and $R_2=x_p\cdots y_q$ be two vertex-disjoint $(V(P_1),V(P_2))$-paths in $G-V(P_3)$. If $R_1$ and $R_2$ are parallel, then there exist two paths $\widetilde{P_1}=P_3P_1[u,x_i]R_1P_2[y_j,y_q]R_2P_1[x_p,x_r]$ and $\widetilde{Q_1}=P_3P_2[u,y_j]R_1P_1[x_i,x_p]R_2P_2[y_q,y_m]$ in $G$ (see Fig. 1$(b)$). If $R_1$ crosses $R_2$, then there also exist two paths $\widetilde{P_2}=P_3P_1[u,x_i]R_1P_2[y_j,y_q]R_2P_1[x_p,x_r]$ and $\widetilde{Q_2}=P_3P_2[u,y_q]R_2P_1[x_p,x_i]R_1P_2[y_j,y_m]$ in $G$ (see Fig. 1($c$)). Now we see that
\begin{center}
$\ell(\widetilde{P_1})+\ell(\widetilde{Q_1})=\ell(\widetilde{P_2})+\ell(\widetilde{Q_2})=2\ell(P_3)+\ell(P_1)+\ell(P_2)+2\ell(R_1)+2\ell(R_2)$
\end{center}
\begin{center}
$~~~~~~~~~~~~\ge 2p-3+2+2>2p,$
\end{center}
this also leads to a contradiction. Thus the lemma follows.
\end{proof}

\begin{lem}\label{lem1}
Let $G$ be a connected graph containing a longest path $L=x_0x_1\ldots x_p$ $(p\ge2)$ and a path $P=y_0y_1\ldots y_s$ $(s\ge p-2)$. Then $|V(L)\cap V(P)|\ge1$ unless $p$ is even, $s=p-2$ and $x_{\frac{p}{2}}y_{\frac{s}{2}}$ is a cut-edge of $G$, where $x_{\frac{p}{2}}\in V(L)$ and $y_{\frac{s}{2}}\in V(P)$.
\end{lem}
\begin{proof}
Suppose to the contrary that $|V(L)\cap V(P)|=0$. Since $G$ is connected, there exists a $(V(L),V(P))$-path $R_1=x\cdots y$, where $x\in V(L)$ and $y\in V(P)$. Obviously, $x$ and $y$ divide $L$ and $P$ into two subpaths, respectively. Without loss of generality, we may assume that  $\ell(L[x_0,x])\ge\ell(L[x,x_p])$ and $\ell(P[y,y_s])\ge\ell(P[y_0,y])$. Then there exists a path $\overline{P}=L[x_0,x]R_1P[y,y_s]$ in $G$ such that
\begin{center}
$\ell(\overline{P})=\ell(L[x_0,x])+\ell(P[y,y_s])+\ell(R_1)\ge\lceil\frac{p}{2}\rceil+\lceil\frac{p-2}{2}\rceil+\ell(R_1).$
\end{center}
If $p$ is odd, then $\ell(\overline{P})\ge p+1$ as $\ell(R_1)\ge1$, this is a contradiction. Hence $p$ is even. Furthermore,  by the maximality of $\ell(L)$, $\ell(L[x_0,x])=\ell(L[x,x_p])=\frac{p}{2}$, $\ell(P[y_0,y])=\ell(P[y,y_s])=\frac{p-2}{2}$, that is $\ell(P)=p-2$, and $\ell(R_1)=1$. This implies that $x=x_{\frac{p}{2}}$, $y=y_{\frac{s}{2}}$ and $R_1=x_{\frac{p}{2}}y_{\frac{s}{2}}$ is an edge of $G$. Therefore, $R_1$ is the only $(V(L),V(P))$-path in $G$, and consequently $R_1$ is a cut-edge.
\end{proof}

\begin{lem}\label{lemz}
Let $G$ be a connected graph containing a longest path $L=x_0x_1\ldots x_p$ and a path $P$ of length at least $p-2$. Suppose $|V(P)\cap V(L)|\ge1$ and $r={\lceil\frac{p}{2}\rceil}$. Then 
\begin{enumerate}
\item[(1)] $x_{r-1}\in V(P)\cap V(L)$ if $|V(P)\cap V(L[x_r,x_p])|=0$, or
\item[(2)] $x_r\in V(P)\cap V(L)$ if $|V(P)\cap V(L[x_0,x_{r-1}])|=0$ and $\ell(P)\ge p-1$, or
\item[(3)] $x_r$ or $x_{r+1}\in V(P)\cap V(L)$ if $|V(P)\cap V(L[x_0,x_{r-1}])|=0$ and $\ell(P)=p-2$.
\end{enumerate}
\end{lem}
\begin{proof}
Suppose not. Since $|V(P)\cap V(L)|\ge1$, It is obvious that $p\ge3$ for (1), $p\ge2$ for (2) and $p\ge4$ for (3). We may first assume that $|V(P)\cap V(L[x_r,x_p])|=0$. Let $i$ be the largest index of vertices in the set of $V(P)\cap V(L[x_0,x_{r-1}])$. Then $x_i$ divides $P$ into two subpaths and one of these paths has at least $\lceil\frac{p-2}{2}\rceil$ edges as $\ell(P)\ge p-2$. Denote this path by $P_1$. Since $V(P_1)\cap V(L[x_i,x_p])=x_i$ and $i\le r-2={\lceil\frac{p}{2}\rceil}-2$, $P_1L[x_i,x_p]$ forms a path in $G$ of length $\lceil\frac{p-2}{2}\rceil+p-\lceil\frac{p}{2}\rceil+2\ge p+1$, contrary to the maximality of $\ell(L)$. We next assume that $|V(P)\cap V(L[x_0,x_{r-1}])|=0$ and  $\ell(P)\ge p-1$. Let $j$ be the smallest index of vertices in the set of $V(P)\cap V(L[x_r,x_p])$. Similarly, $x_j$ divides $P$ into two subpaths and one of these paths has at least $\lceil\frac{p-1}{2}\rceil$ edges. Denote this path by $P_2$. Since $V(P_2)\cap V(L[x_0,x_j])=x_j$ and $j\ge r+1={\lceil\frac{p}{2}\rceil}+1$, $L[x_0,x_j]P_2$ forms a path in $G$ of length $\lceil\frac{p}{2}\rceil+1+\lceil\frac{p-1}{2}\rceil\ge p+1$, yielding a contradiction. Finally, we assume that $|V(P)\cap V(L[x_0,x_{r-1}])|=0$ and $\ell(P)=p-2$. Define $x_j$ and $P_2$ as above. Then $\ell(P_2)\ge\lceil\frac{p-2}{2}\rceil$. As $j\ge r+2={\lceil\frac{p}{2}\rceil}+2$, $L[x_0,x_j]P_2$ forms a path in $G$ of length $\lceil\frac{p}{2}\rceil+2+\lceil\frac{p-2}{2}\rceil\ge p+1$, yielding a contradiction again. Thus the lemma follows.
\end{proof}

\begin{lem}\label{lem2}
Let $G$ be a connected graph of order $n\ge4$ with longest path of length $p$. Let $L_1,L_2$ and $L_3$ be three paths of $G$ such that any two of them intersect only at the endpoint $u$. Suppose $\ell(L_1)=\lceil\frac{p}{2}\rceil$, $\ell(L_2)=p-\lceil\frac{p}{2}\rceil$ and $\ell(L_3)\ge \lceil\frac{p}{2}\rceil-1$. Then $u$ is a cut-vertex.
\end{lem}
\begin{proof}
If $p=2$, then $G$ is a star with center $u$. Clearly, $u$ is a cut-vertex. Therefore, we assume that $p\ge3$ in the following proof. Notice that 
\begin{center}
$2\ell(L_2)+\ell(L_1)+\ell(L_3)\ge2(p-\lceil\frac{p}{2}\rceil)+\lceil\frac{p}{2}\rceil+\lceil\frac{p}{2}\rceil-1=2p-1$ 
\end{center}
and 
\begin{center}
$2\ell(L_1)+\ell(L_2)+\ell(L_3)\ge2\cdot\lceil\frac{p}{2}\rceil+p-\lceil\frac{p}{2}\rceil+\lceil\frac{p}{2}\rceil-1\ge2p-1$. 
\end{center}
By Lemma \ref{lem}, there is no $(V(L_3),V(L_1\cup L_2))$-path in $G-\{u\}$, which implies that $u$ is a cut-vertex.
\end{proof}

\begin{lem}\label{lemqq}
Let $G$ be a graph with a Hamilton path $L=x_0x_1\ldots x_p$. Then there is a bond meeting all paths of length at least $\lceil\frac{p+1}{2}\rceil$.
\end{lem}
\begin{proof}
Let $X=\{x_0,\ldots,x_{\lceil \frac{p}{2}\rceil-1}\}$ and $Y=\{x_{\lceil \frac{p}{2}\rceil},\ldots,x_p\}$. Then $X\cup Y=V(G)$ as $L$ is a Hamilton path. Define $B$ as the set of edges of $G$ with one end in $X$ and the other in $Y$. Clearly, $B$ is a bond of $G$. Notice that $\lceil \frac{p}{2}\rceil-1<\lceil\frac{p+1}{2}\rceil$ and $p-\lceil \frac{p}{2}\rceil<\lceil\frac{p+1}{2}\rceil$. Thus there is no path of length $\lceil\frac{p+1}{2}\rceil$ in $G[X]$ or $G[Y]$. Therefore, all paths of length at least $\lceil\frac{p+1}{2}\rceil$ meet the bond $B$ in $G$.
\end{proof}

We finally list a helpful lemma and two known theorems that shall be applied in later proofs.
\begin{lem}[\cite{B}]\label{lemfan}
Let $G$ be a $k$-connected graph, and let $X$ and $Y$ be disjoint subsets of $V(G)$ such that $|X|\ge k$ and $|Y|\ge k$. Then there exist $k$ vertex-disjoint $(X,Y)$-paths in $G$.
\end{lem}
\begin{thm}[Dirac\cite{D}]\label{D}
If $G$ is a 2-connected graph, then $G$ contains a cycle of length at least $min\{2\delta(G),|V(G)|\}$. 
\end{thm}
\begin{thm}[Erd\H{o}s-Szekeres on Monotone Subsequence \cite{ES}]\label{thmez}
For any $n\ge2$, every sequence $(a_1,a_2,\ldots,a_N)$ of real numbers, with $N\ge (n-1)^2+1$, contains a monotone subsequence of length $n$; more precisely, there are indices $i_1<i_2<\cdots<i_n$ such that either $a_{i_1}\le\cdots\le a_{i_n}$ or $a_{i_1}\ge\cdots\ge a_{i_n}$. 
\end{thm}
\section{Proofs of main results}
\textbf{Proof of Theorem \ref{thm1}.} Let $L=x_0x_1\ldots x_p$ be a longest path in $G$. Set $r=\lceil \frac{p}{2}\rceil$. Let $S$ be a vertex set containing $\{x_r,x_{r+1},\ldots,x_p\}$. Let $G_1$ denote a component of $G-S$ containing $\{x_0,x_1,\ldots, x_{r-1}\}$. Define $B$ as the set of edges of $G$ with one end in $V(G_1)$ and the other in $S$. Obviously, $B$ is a bond of $G$. Let $\mathscr{L}$ be the set of all paths of length at least $p-1$ in $G$. If each path of $\mathscr{L}$ has a vertex in $V(G_1)$ and another in $S$, then $B$ meets all paths of $\mathscr{L}$, and hence statement 1 holds. Now we assume that $B$ does not meet all paths of $\mathscr{L}$. Then there exists a path $L^{'}$ in $\mathscr{L}$ such that  $V(L^{'})\subseteq V(G_1)$ or $V(L^{'})\subseteq V(G)-V(G_1)$. Furthermore, $|V(L)\cap V(L^{'})|\ge1$ by Lemma \ref{lem1}.

We first assume that $p$ is even. If $V(L^{'})\subseteq V(G_1)$, then by Lemma \ref{lemz}(1), $x_{r-1}\in V(L^{'})\cap V(L)$. Clearly, $x_{r-1}$ divides $L^{'}$ into two subpaths and one of these paths has at least $\frac{p}{2}$ edges. Denote this path by $Q$. Notice that $V(Q)\cap V(L[x_{r-1},x_p])=x_{r-1}$ and $\ell(L[x_{r-1},x_p])=\frac{p}{2}+1$. Thus $QL[x_{r-1},x_p]$ forms a path in $G$ of length $\frac{p}{2}+\frac{p}{2}+1=p+1$, yielding a contradiction. Therefore, $V(L^{'})\subseteq V(G)-V(G_1)$. Also, by Lemma \ref{lemz}(2), $x_r\in V(L)\cap V(L^{'})$. Similarly, $x_r$ divides $L^{'}$ into two subpaths, denoted by $Q_1$ and $Q_2$. Then $Q_1,Q_2$ and $L[x_0,x_r]$ are three paths of $G$ and any two of them intersect only at the vertex $x_r$. As $\ell(L[x_0,x_r])=\frac{p}{2}$, without loss of generality, we have $\ell(Q_1)=\frac{p}{2}$ and $\ell(Q_2)\ge \frac{p}{2}-1$. It follows that $x_r$ is a cut-vertex passing through all paths of length at least $p-1$ by Lemma \ref{lem2}.

We next assume that $p$ is odd. Again, if $V(L^{'})\subseteq V(G_1)$, then  $x_{r-1}\in V(L)\cap V(L^{'})$ and thus $x_{r-1}$ divides $L^{'}$ into two subpaths, denoted by $P_1$ and $P_2$. Now $P_1,P_2$ and $L[x_{r-1},x_p]$ are three paths of $G$ and any two of them intersect only at the vertex $x_{r-1}$. As $\ell(L[x_{r-1},x_p])=\frac{p+1}{2}$, by the maximality of $\ell(L)$, we have $\ell(P_1)=\ell(P_2)=\frac{p-1}{2}$, that is $\ell(L^{'})=p-1$. It follows that $x_{r-1}$ is a cut-vertex passing through all paths in $G_1$ of length $p-1$ by Lemma \ref{lem2}. If $V(L^{'})\subseteq V(G)-V(G_1)$, then by symmetry, we see that $\ell(L^{'})=p-1$ and $x_r$ is a cut-vertex passing through all paths in $G-V(G_1)$ of length $p-1$. Therefore, we conclude that $x_r$ or $x_{r-1}$ is a cut-vertex passing through all paths of length at least $p-1$. Hence statement 3 holds. In particular, if both $G_1$ and $G-V(G_1)$ contain paths of length $p-1$, then it is easily seen that $x_{r-1}x_r$ is a cut-edge (one-element bond) meeting all longest paths. Also, $\{x_{r-1},x_r\}$ meets all paths of length at least $p-1$. Thus statement 2 follows. 

This completes the proof of Theorem \ref{thm1}.\qed

Let $G$ be a triangle. It is easy to check that none of the statements in Theorem \ref{thm1} holds. Therefore, $n\ge4$ is the best possible for Theorem \ref{thm1}.

\textbf{Proof of Theorem \ref{thm2}.} Note that $\delta(G)\ge3$ as $G$ is 3-connected. Then by Lemmas \ref{lemqq} and \ref{D}, the statement is clearly true when $6\le n\le 7$ as $p-2\ge \lceil\frac{p+1}{2}\rceil$. We next prove that the statement is also true when $n\ge8$. If $p=6$, let $C$ be a longest cycle of $G$. Then by Lemma \ref{D}, $|C|=6$. Let $C=v_1v_2v_3v_4v_5v_6v_1$ and $U=V(G)-V(C)$. Then $|U|\ge 2$ as $n\ge 8$. Set $U=\{u_1,u_2,\cdots, u_{n-6}\}$. Since $G$ is $3$-connected and $p=6$, $U$ is an independent set. Furthermore, any vertex in $U$ cannot have consecutive neighbors on $C$ as $C$ is a longest cycle of $G$. Since $\delta(G)\ge 3$ and $p=6$, it is easy to derive, without loss of generality, that $N(u_i)=\{v_1,v_3,v_5\}$ for any $u_i\in U$. Moreover, if there exists an edge in $G[\{v_2,v_4,v_6\}]$, by symmetry we say $v_2v_4$, then $v_1u_1v_3v_2v_4v_5v_6v_1$ forms a cycle of length 7, yielding a contradiction. Thus $U\cup \{v_2,v_4,v_6\}$ is an independent set of $G$. Define $B^1$ as the set of edges of $G$ with one end in $\{u_1,v_1,v_3\}$ and the other in $V(G)-\{u_1,v_1,v_3\}$. Since both $G[\{u_1,v_1,v_3\}]$ and $G-\{u_1,v_1,v_3\}$ are connected, it follows that $B^1$ is a bond of $G$ that meets all paths of length 4. Therefore, we shall assume that $p\ge7$ in the following proof.

Suppose to the contrary that $G$ does not have such a bond meeting all paths of length at least $p-2$. Let $L=x_0x_1\ldots x_p$ be a longest path in $G$. Set $r=\lceil \frac{p}{2}\rceil$. Let $S$ be a vertex set containing $\{x_r,x_{r+1},\ldots,x_p\}$. Let $G_1$ denote a component of $G-S$ containing $\{x_0,x_1,\ldots, x_{r-1}\}$. Define $B^2$ as the set of edges of $G$ with one end in $V(G_1)$ and the other in $S$. Clearly, $B^2$ is a bond of $G$. Let $\mathscr{L}$ be the set of all paths of length at least $p-2$ in $G$. Notice that if each path of $\mathscr{L}$ has a vertex in $V(G_1)$ and another in $S$, then $B^2$ meets all paths of $\mathscr{L}$. Thus there exists a path $L^{'}$ in $\mathscr{L}$ such that  $V(L^{'})\subseteq V(G_1)$ or $V(L^{'})\subseteq V(G)-V(G_1)$. We distinguish two cases.
\begin{case}\label{case1}
 $p$ is odd.		
\end{case}
Assume that $V(L^{'})\subseteq V(G_1)$. By Lemmas \ref{lem1} and \ref{lemz}(1), we have $x_{r-1}\in V(L)\cap V(L^{'})$. Then $x_{r-1}$ divides $L^{'}$ into two subpaths, denoted by $P$ and $Q$. Thus $P,Q$ and $L[x_{r-1},x_p]$ are three paths of $G$ and any two of them intersect only at the vertex $x_{r-1}$. Notice that $\ell(L[x_{r-1},x_p])=\frac{p+1}{2}$. By the maximality of $\ell(L)$, without loss of generality, we have $\ell(P)=\frac{p-1}{2}$, $\frac{p-1}{2}\ge\ell(Q)\ge\frac{p-3}{2}$. It follows that
\begin{center}
$2\ell(L[x_{r-1},x_p])+\ell(P)+\ell(Q)\ge2\cdot\frac{p+1}{2}+\frac{p-1}{2}+\frac{p-3}{2}=2p-1$, 
\end{center}
and 
\begin{center}
$2\ell(P)+\ell(L[x_{r-1},x_p])+\ell(Q)\ge 2\cdot\frac{p-1}{2}+\frac{p+1}{2}+\frac{p-3}{2}=2p-2$.
\end{center}
Therefore by Lemma \ref{lem}, there are no two vertex-disjoint $(V(Q),V(P\cup L[x_{r-1},x_p]))$-paths in $G-\{x_{r-1}\}$, which contradicts the fact that $G$ is 3-connected. Since $p$ is odd, by symmetry, the proof for the case $V(L^{'})\subseteq V(G)-V(G_1)$ is similar. Thus we complete the proof of Case 1.
\begin{case}\label{case2}
$p$ is even.		
\end{case}
We first assume that $V(L^{'})\subseteq V(G_1)$. Again, by Lemmas \ref{lem1} and \ref{lemz}(1), we have $x_{r-1}\in V(L^{'})\cap V(L)$. Define $P$ and $Q$ as in Case \ref{case1}. Since $\ell(L[x_{r-1},x_p])=\frac{p+2}{2}$, by the maximality of $\ell(L)$, we have $\ell(P)=\ell(Q)=\frac{p-2}{2}$. It follows that
\begin{center}
$2\ell(L[x_{r-1},x_p])+\ell(P)+\ell(Q)=2\cdot\frac{p+2}{2}+\frac{p-2}{2}+\frac{p-2}{2}=2p$,
\end{center}
and 
\begin{center}
$2\ell(P)+\ell(L[x_{r-1},x_p])+\ell(Q)=2\cdot\frac{p-2}{2}+\frac{p+2}{2}+\frac{p-2}{2}=2p-2$.
\end{center}
By Lemma \ref{lem}, there are no two vertex-disjoint $(V(Q),V(P\cup L[x_{r-1},x_p]))$-paths in $G-\{x_{r-1}\}$, this contradicts  the fact that $G$ is 3-connected.

We next assume that $V(L^{'})\subseteq V(G)-V(G_1)$. By Lemmas \ref{lem1}, \ref{lemz}(2) and \ref{lemz}(3), we have $x_r$ or $x_{r+1}\in V(L^{'})\cap V(L)$. In either case, let $u$ denote the common vertex. Similarly, $u$ divides $L^{'}$ into two subpaths, denoted by $P$ and $Q$. Let $R$ be a subpath of $L$ containing $\{x_0,x_1,\ldots, u\}$. Then $P,Q$ and $R$ are three paths of $G$ and any two of them intersect only at the vertex $u$. Notice that $\ell(R)\ge\frac{p}{2}$. By the maximality of $\ell(L)$, we see that either (1) $\ell(P)\ge\frac{p-2}{2}$ and $\ell(Q)\ge\frac{p-2}{2}$ or (2) $\ell(P)=\frac{p}{2}$ and $\ell(Q)\ge\frac{p-4}{2}$. We first prove the following claim.
\begin{claim}\label{claim1}
For (1), there are no two vertex-disjoint $(V(R),V(P))$-paths in $G-V(Q)$. For both (1) and (2), there are no two vertex-disjoint $(V(Q),V(P))$ $(resp.$, $(V(Q),V(R)))$-paths in $G-V(R)$ $(resp.$, $G-V(P))$.  
\end{claim}
\begin{proof}
For (1), we have
\begin{center}
$2\ell(Q)+\ell(R)+\ell(P)\ge2\cdot\frac{p-2}{2}+\frac{p}{2}+\frac{p-2}{2}=2p-3$.
\end{center}
For both (1) and (2), we have
\begin{center}
$2\ell(R)+\ell(Q)+\ell(P)\ge2\cdot\frac{p}{2}+p-2=2p-2$,
\end{center}
and 
\begin{center}
$2\ell(P)+\ell(Q)+\ell(R)\ge \frac{p-2}{2}+p-2+\frac{p}{2}=2p-3$.
\end{center}
Therefore by Lemma \ref{lem}, Claim 1 holds.
\end{proof} 

For the convenience of the proofs, let $P=a_0a_1\ldots u$ and $Q=b_0b_1\ldots u$. Set $A=V(P)-\{u\}, B=V(Q)-\{u\}$ and $X=V(R)-\{u\}$. We next consider the following two subcases based on the fact that $p$ is even and $V(L^{'})\subseteq V(G)-V(G_1)$.
\begin{subcase}\label{subcase2.1}
$\ell(P)\ge\frac{p-2}{2}$, $\ell(Q)\ge\frac{p-2}{2}$ and $\ell(R)\ge\frac{p}{2}$.
\end{subcase}
\begin{center}
	\includegraphics[scale=.65]{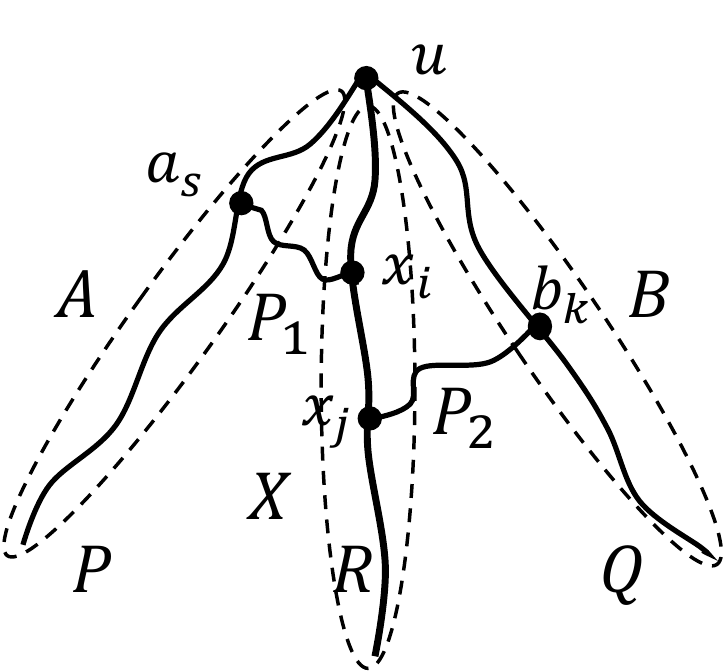}
\end{center}
\begin{center}
	Fig. 2. Illustration of $P,R$ and $Q$ along with $P_1$ and $P_2$.
\end{center}

Since $p\ge8$, it is clear that $|A|\ge3$, $|B|\ge3$ and $|X|\ge4$. As $G$ is 3-connected, by Lemma \ref{lemfan} and Claim \ref{claim1}, there exist two vertex-disjoint paths: $(X,A)$-path $P_1=x_i\cdots a_s$ and $(X,B)$-path $P_2=x_j\cdots b_k$ in $G-\{u\}$. Without loss of generality, we say $i>j$ (see Fig. 2). We first prove the following three properties.
\begin{itemize}
\item[($p_1$)] $\ell(P[a_s,u])\le\ell(R[x_i,u])\le\ell(P[a_s,u])+1$.
\item[($p_2$)] $\ell(Q[b_k,u])\le\ell(R[x_j,u])\le\ell(Q[b_k,u])+1$.
\item[($p_3$)] If there is a $(A,B)$-path $P^{\Delta}=a\cdots b$ in $G-V(R)$, where $a\in A$ and $b\in B$, then $\ell(P[u,a])=\ell(Q[u,b])$.
\end{itemize}
\begin{proof}
For property ($p_1$), if $\ell(P[a_s,u])>\ell(R[x_i,u])$, then $R[x_0,x_i]P_1P[a_s,u]Q$ forms a path of length at least $p+1$, a contradiction. If $\ell(R[x_i,u])>\ell(P[a_s,u])+1$, then $P[a_0,a_s]P_1R[x_i,u]Q$ forms a path of length at least $p+1$, a contradiction again. So property ($p_1$) holds, and property ($p_2$) follows by similar arguments. As for property ($p_3$), without loss of generality, assume that $\ell(P[u,a])<\ell(Q[u,b])$. Then $P[a_0,a]P^{\Delta}Q[b,u]R$ forms a path of length at least $p+1$, yielding a contradiction. Thus property ($p_3$) holds as well.
\end{proof}

From properties ($p_1$) and ($p_2$), we see that $\ell(Q[u,b_k])\ge\ell(P[u,a_s])$. Furthermore, $\ell(P_1)=\ell(P_2)=1$, otherwise $P[a_0,a_s]P_1R[x_i,u]Q[u,b_k]P_2R[x_j,x_0]$ forms a path of length at least $p+1$. Set $A_1=V(P[a_0,a_{s-1}])$, $A_2=V(P[a_{s+1},u])-\{u\}$, $B_1=V(Q[b_0,b_{k-1}])$ and $B_2=V(Q[b_{k+1},u])-\{u\}$. Consider $P^{\Delta}$ defined in ($p_3$) again. We next show that if one of following three conditions holds, then $G$ contains a path of length at least $p+1$.
\begin{itemize}
\item[(I)] $a\in A_2$;
\item[(II)] $a\in A_1$ and $b\in B_1$;
\item[(III)] $a\in A_1$ and $b\in B_2$.
\end{itemize}
\begin{center}
\includegraphics[scale=.65]{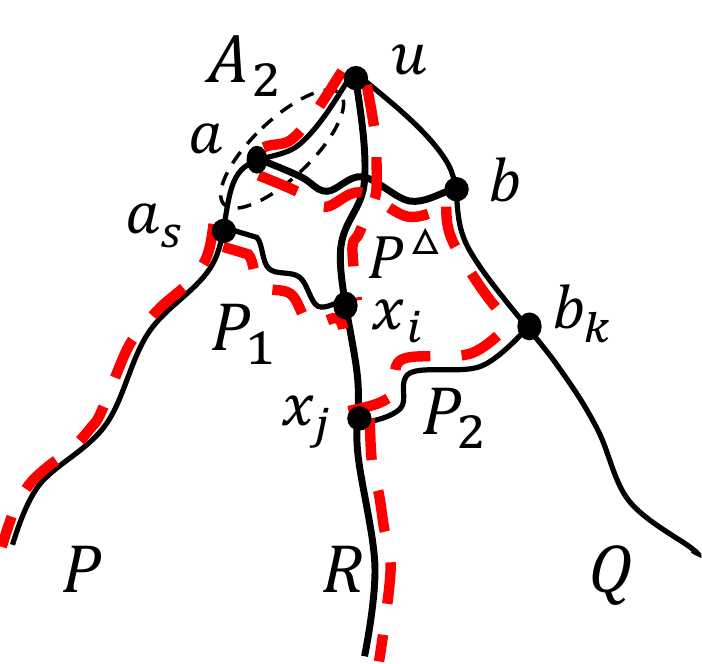}
\hfil
\includegraphics[scale=.65]{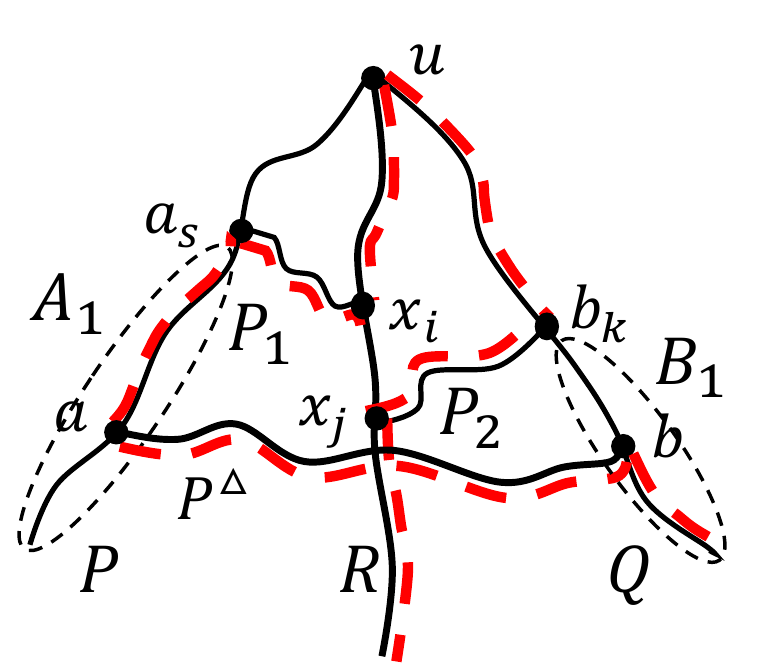}

\hfil
($a$)~~~~~~~~~
\hfil
~~~~~~~~~~~~~~~~~~~~~~~~($b$)~~~~~~~~~~~~~~~~~~
\end{center}
\begin{center}
\includegraphics[scale=.65]{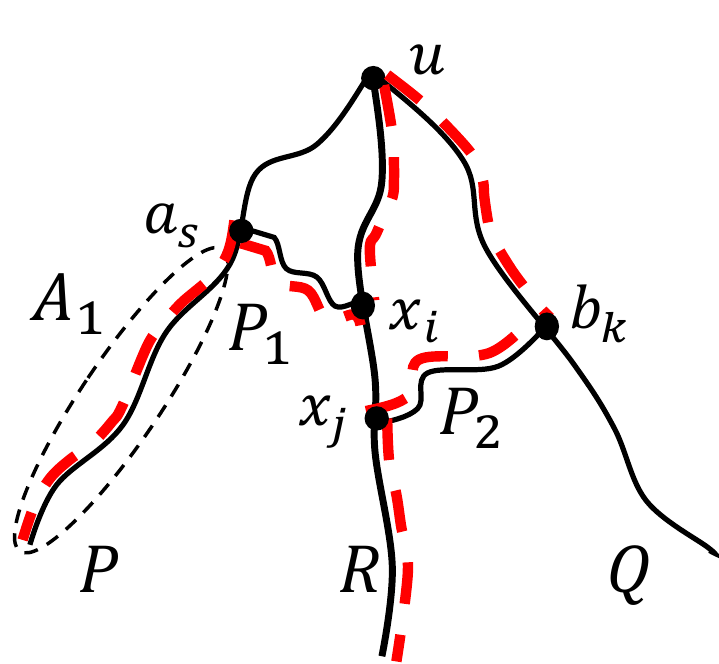}
\hfil
\includegraphics[scale=.65]{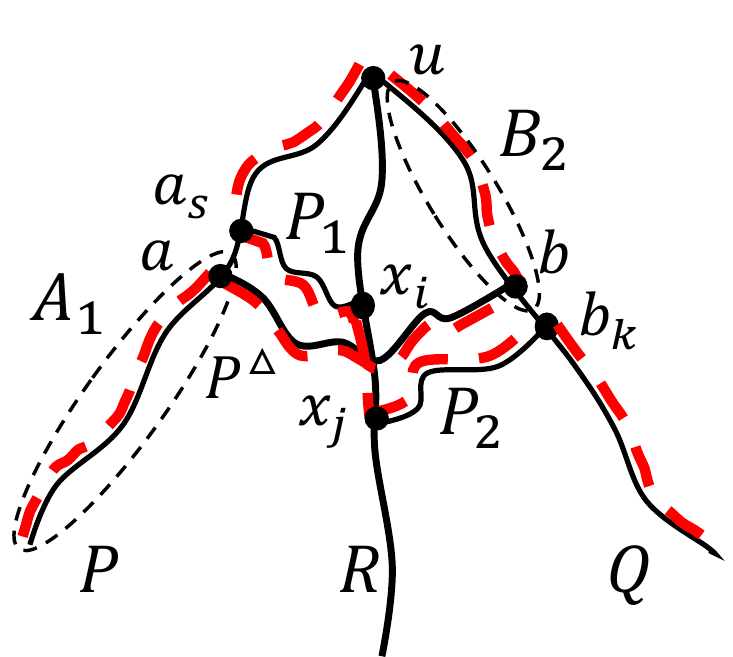}

\hfil
($c$)~~~~~~~~~
\hfil
~~~~~~~~~~~~~~~~~~~~~~~($d$)~~~~~~~~~~~~~~~
\end{center}
\begin{center}
Fig. 3. Configurations of $\overline{P_1}$, $\overline{P_2}$, $\overline{P_3}$ and $\overline{P_4}$.
\end{center}
\begin{proof}
If condition (I) holds, then by property ($p_3$), we have $b\in B_2$. Let 
\begin{center}
$\overline{P_1}=P[a_0,a_s]P_1R[x_i,u]P[u,a]P^{\Delta}Q[b,b_k]P_2R[x_j,x_0]$
\end{center}
be a path in $G$ (see Fig. 3($a$)). Thus by properties ($p_1$)-($p_3$), we have
\begin{center}
$\ell(\overline{P_1})=\ell(P[a_0,a_s])+\ell(P_1)+\ell(R[x_i,u])+\ell(P[u,a])+\ell(P^{\Delta})+\ell(Q[b,b_k])$
\end{center}
\begin{center}
$+\ell(P_2)+\ell(R[x_j,x_0])~~~~~~~~~~~~~~~~~~~~~~~~~~~~~~~~~~~~~~~~~~~~~~~~~~$
\end{center}
\begin{center}
$~~~~~~~~~\ge\ell(P[a_0,a_s])+ \ell(P[a_s,u])+\ell(Q[u,b])+\ell(Q[b,b_k])+\ell(R[x_j,x_0])+3$
\end{center}
\begin{center}
$\ge\ell(P[a_0,a_s])+ \ell(P[a_s,u])+\ell(R[u,x_j])+\ell(R[x_j,x_0])+2~~~~~$
\end{center}
\begin{center}
$\ge\frac{p-2}{2}+\frac{p}{2}+2=p+1,~~~~~~~~~~~~~~~~~~~~~~~~~~~~~~~~~~~~~~~~~~~~~~~~~~$
\end{center}
as desired. If condition (II) holds, then let 
\begin{center}
$\overline{P_2}=R[x_0,x_j]P_2Q[b_k,u]R[u,x_i]P_1P[a_s,a]P^{\Delta}Q[b,b_0]$.
\end{center}
be a path in $G$ (see Fig. 3($b$)). Thus by properties ($p_1$)-($p_3$), we have
\begin{center}
$\ell(\overline{P_2})=\ell(R[x_0,x_j])+\ell(P_2)+\ell(Q[b_k,u])+\ell(R[u,x_i])+\ell(P_1)+\ell(P[a_s,a])$
\end{center}
\begin{center}
$+\ell(P^{\Delta})+\ell(Q[b,b_0])~~~~~~~~~~~~~~~~~~~~~~~~~~~~~~~~~~~~~~~~~~~~~~~~~~$
\end{center}
\begin{center}
$~~~~~~~~\ge\ell(R[x_0,x_j])+\ell(R[x_j,u])+\ell(P[u,a_s])+\ell(P[a_s,a])+\ell(Q[b,b_0])+2$
\end{center}
\begin{center}
$=\ell(R[x_0,x_j])+\ell(R[x_j,u])+\ell(Q[u,b])+\ell(Q[b,b_0])+2~~~~~~~~~~$
\end{center}
\begin{center}
$\ge\frac{p}{2}+\frac{p-2}{2}+2=p+1,~~~~~~~~~~~~~~~~~~~~~~~~~~~~~~~~~~~~~~~~~~~~~~~~~~~$
\end{center}
as desired. Now, we consider condition (III). If $\ell(R[x_i,u])=\ell(P[a_s,u])+1$ and $\ell(R[u,x_j])=\ell(Q[u,b_k])$, then let 
\begin{center}
$\overline{P_3}=P[a_0,a_s]P_1R[x_i,u]Q[u,b_k]P_2R[x_j,x_0]$
\end{center}
be a path in $G$ (see Fig. 3($c$)). Otherwise, without loss of generality, we may assume that $\ell(R[x_i,u])=\ell(P[a_s,u])$ and let 
\begin{center}
$\overline{P_4}=P[a_0,a]P^{\Delta}Q[b,u]P[u,a_s]P_1L_1[x_i,x_j]P_2Q[b_k,b_0]$
\end{center}
be a path in $G$ (see Fig. 3($d$)). Thus
\begin{center}
$\ell(\overline{P_3})=\ell(P[a_0,a_s])+\ell(P_1)+\ell(R[x_i,u])+\ell(Q[u,b_k])+\ell(P_2)+\ell(R[x_j,x_0])$
\end{center}
\begin{center}
$\ge\ell(P[a_0,a_s])+\ell(P[a_s,u])+\ell(R[u,x_j])+\ell(R[x_j,x_0])+3~~~~~~~~$
\end{center}
\begin{center}
$=\frac{p-2}{2}+\frac{p}{2}+3=p+2,~~~~~~~~~~~~~~~~~~~~~~~~~~~~~~~~~~~~~~~~~~~~~~~~~~~~~$
\end{center}
and, by properties ($p_1$)-($p_3$), 
\begin{center}
$\ell(\overline{P_4})=\ell(P[a_0,a])+\ell(P^{\Delta})+\ell(Q[b,u])+\ell(P[u,a_s])+\ell(P_1)+\ell(R[x_i,x_j])$
\end{center}
\begin{center}
$+\ell(P_2)+\ell(Q[b_k,b_0])~~~~~~~~~~~~~~~~~~~~~~~~~~~~~~~~~~~~~~~~~~~~~~~~~~$
\end{center}
\begin{center}
$~~~~~~~~~\ge\ell(P[a_0,a])+\ell(P[a,u])+\ell(P[u,a_s])+\ell(R[x_i,x_j])+\ell(Q[b_k,b_0])+3$
\end{center}
\begin{center}
$~~~~~~~~~=\ell(P[a_0,a])+\ell(P[a,u])+\ell(R[u,x_i])+\ell(R[x_i,x_j])+\ell(Q[b_k,b_0])+3$
\end{center}
\begin{center} 
$\ge\ell(P[a_0,a])+\ell(P[a,u])+\ell(Q[u,b_k])+\ell(Q[b_k,b_0])+3~~~~~~~~$
\end{center}
\begin{center}
$\ge\frac{p-2}{2}+\frac{p-2}{2}+3=p+1,~~~~~~~~~~~~~~~~~~~~~~~~~~~~~~~~~~~~~~~~~~~~~~~$
\end{center}
as desired.
\end{proof}

Finally, we get to the heart of the proof for Subcase \ref{subcase2.1}. As $G$ is 3-connected, there exists an $(A,X\cup B)$-path $P_3=a_t\cdots v$ in $G-\{a_s,u\}$, where $a_t\in A$. If $v\in B$, then, to avoid conditions (I)-(III), $a_t\in A_1$ and $v=b_k$. Similarly, there exists a $(B,X\cup A)$-path $P^{'}=b^{'}_l\cdots z^{'}$ in $G-\{b_k,u\}$, where $b^{'}_l\in B$. When  $z^{'}\in A$, to avoid conditions (I)-(III), we have $z^{'}=a_s$. However, this contradicts Claim \ref{claim1} as $P^{'}$ and $P_3$ are two vertex-disjoint $(A,B)$-paths. It follows that $z^{'}\in X$, and by Claim \ref{claim1}, $z^{'}=x_j$.  Now we can consider $P^{'}$ as $P_2$. Then it will satisfy condition (II) or (III), which in turn gives us a path of length at least $p+1$, yielding a contradiction. Hence $v\in X$, and by Claim \ref{claim1}, $v=x_i$. 

Again, as $G$ is 3-connected, there exists a $(B,X\cup A)$-path $P_4=b_l\cdots z$ in $G-\{b_k,u\}$, where $b_l\in B$. By symmetry, we have $z=x_j$. Moreover, there also exists a $(B,X\cup A)$-path $P_5=b_p\cdots w$ in $G-\{x_j,u\}$, where $b_p\in B$. By similar arguments as to $P_4$, we have $w\in X$. Now $P_5$ and one of $P_2$ and $P_4$ are two vertex-disjoint $(B,X)$-paths, this contradicts Claim \ref{claim1}. Thus we complete the proof of Subcase \ref{subcase2.1}.

\begin{subcase}\label{subcase2.2}
$\ell(P)=\frac{p}{2}$, $\ell(Q)\ge\frac{p-4}{2}$ and $\ell(R)=\frac{p}{2}$.
\end{subcase}
\begin{center}
	\includegraphics[scale=.65]{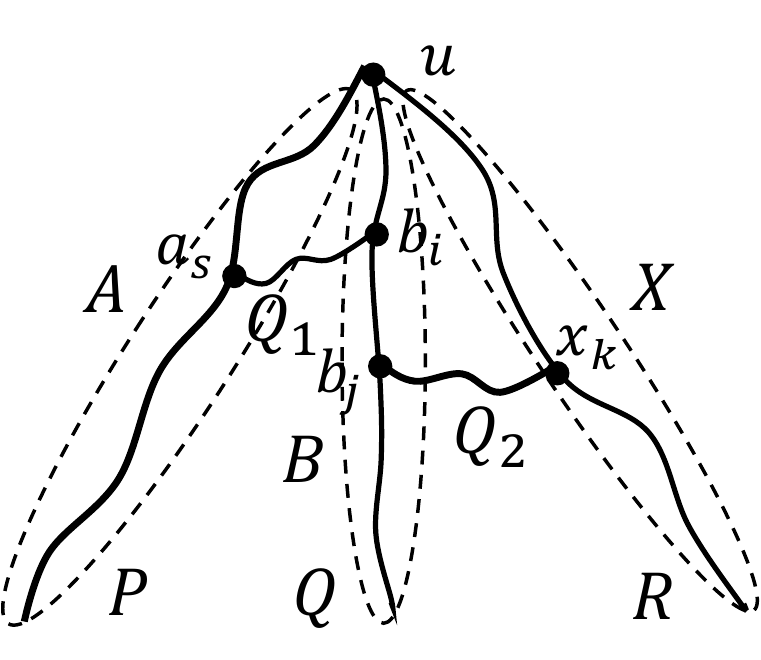}
\end{center}
\begin{center}
	Fig. 4. Illustration of $P,Q$ and $R$ along with $Q_1$ and $Q_2$.
\end{center}

Since $p\ge8$, it is easily seen that $|A|\ge4,|B|\ge2$ and $|X|\ge4$. As $G$ is 3-connected, by Lemma \ref{lemfan} and Claim \ref{claim1}, there exist two vertex-disjoint paths: $(B,A)$-path $Q_1=b_i\cdots a_s$ and $(B,X)$-path $Q_2=b_j\cdots x_k$ in $G-\{u\}$. Without loss of generality, we say $i>j$ (see Fig. 4). We first prove the following two properties.
\begin{itemize}
\item[($p_4$)] $\ell(P[u,a_s])=\ell(Q[u,b_i])+1$ and $\ell(R[u,x_k])=\ell(Q[u,b_j])+1$. Thus $\ell(Q_1)=1$ and $\ell(Q_2)=1$.
\item[($p_5$)] If there is a $(A,X)$-path $Q^{\Delta}=a\cdots x$ in $G-V(Q)$, where $a\in A$ and $x\in X$, then $|\ell(P[u,a])-\ell(R[u,x])|\le1$.
\end{itemize}
\begin{proof}
For property ($p_4$), without loss of generality, assume that $\ell(P[u,a_s])\ge\ell(Q[u,b_i])+2$. Then $Q[b_0,b_i]Q_1P[a_s,u]R$ forms a path of length at least $p+1$, which is impossible. So $\ell(P[u,a_s])=\ell(Q[u,b_i])+1$, and $\ell(R[u,x_k])=\ell(Q[u,b_j])+1$ follows by similar arguments. Furthermore, by the maximality of $\ell(L)$, we have $\ell(Q_1)=1$ and $\ell(Q_2)=1$. For property ($p_5$), without loss of generality, if $\ell(P[u,a])\ge\ell(R[u,x])+2$, then $QP[u,a]Q^{\Delta}R[x,x_0]$ forms a path length $p+1$, yielding a contradiction and thus property ($p_5$) holds. 
\end{proof}

From property ($p_4$), we see that $\ell(R[u,x_k])>\ell(P[u,a_s])$. Set $A_1=V(P[a_0,a_{s-1}])$, $A_2=V(P[a_{s+1},u])-\{u\}$, $X_1=V(R[x_0,x_{k-1}])$ and $X_2=V(R[x_{k+1},u])-\{u\}$. Consider $Q^{\Delta}$ defined in ($p_5$) again. We next show that if one of following two conditions holds, then $G$ contains a path of length at least $p+1$.
\begin{itemize}
\item[(IV)] $a\in A_2$;
\item[(V)] $a\in A_1$ and $x\in X_1$.
\end{itemize}
\begin{center}
	\includegraphics[scale=.6]{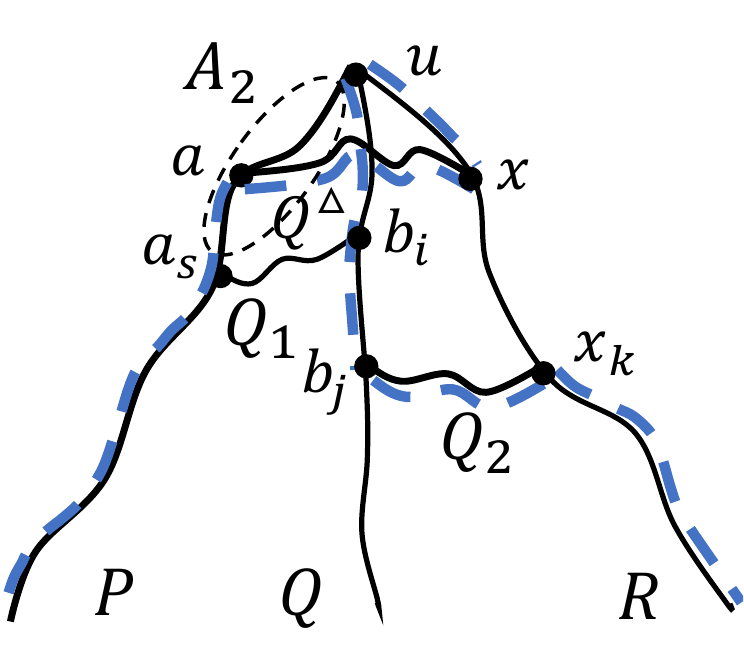}
	\hfil
	\includegraphics[scale=.6]{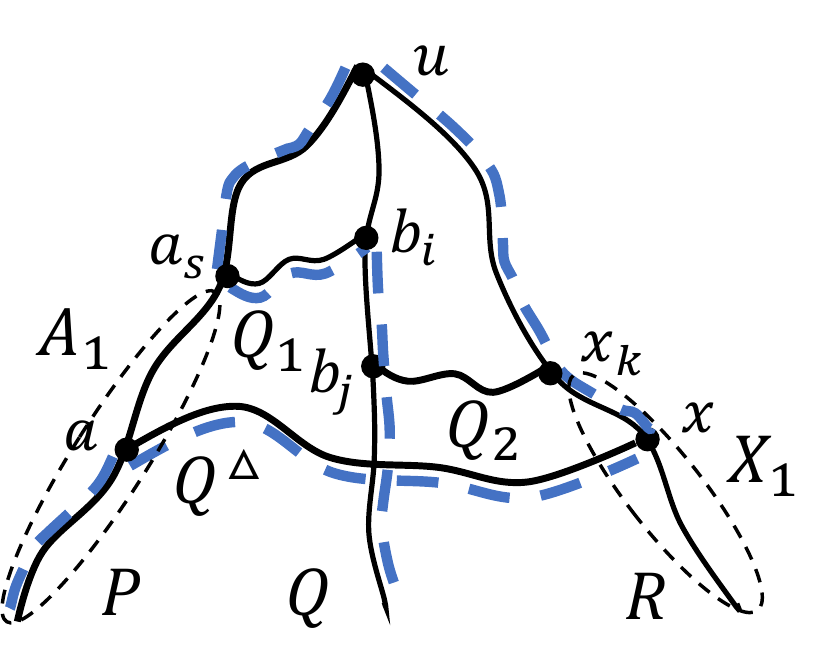}
	
\hfil
($a$)~~~~~~~~~
\hfil
~~~~~~~~~~~~~~~~~~~~~~~~($b$)~~~~~~~~~~~~~~~~~~~~	
\end{center}
\begin{center}
	Fig. 5. Configurations of $\overline{Q_1}$ and $\overline{Q_2}$.
\end{center}
\begin{proof}
Without loss of generality,  we may always assume that $\ell(R[u,x])\ge\ell(P[u,a])$. If condition (IV) holds, then by property ($p_5$), we have $x\in X_2$. Let 	
\begin{center}
$\overline{Q_1}=P[a_0,a]Q^{\Delta}R[x,u]Q[u,b_j]Q_2R[x_k,x_0]$.
\end{center}	 
be a path in $G$ (see Fig. 5($a$)). Thus by property ($p_4$), we have
\begin{center}
$\ell(\overline{Q_1})=\ell(P[a_0,a])+\ell(Q^{\Delta})+\ell(R[x,u])+\ell(Q[u,b_j])+\ell(Q_2)+\ell(R[x_k,x_0])$
\end{center}	
\begin{center}
$~\ge\ell(P[a_0,a])+\ell(P[a,u])+\ell(R[u,x_k])+\ell(R[x_k,x_0])+1~~~~~~~~~~$
\end{center}	
\begin{center}
$=\frac{p}{2}+\frac{p}{2}+1=p+1,~~~~~~~~~~~~~~~~~~~~~~~~~~~~~~~~~~~~~~~~~~~~~~~~~~~~~~$
\end{center}
as desired. If condition (V) holds, then let 	
\begin{center}
$\overline{Q_2}=P[a_0,a]Q^{\Delta}R[x,u]P[u,a_s]Q_1Q[b_i,b_0]$
\end{center}	
be a path in $G$ (see Fig. 5($b$)). Again, by property ($p_4$), we have	
\begin{center}
$\ell(\overline{Q_2})=\ell(P[a_0,a])+\ell(Q^{\Delta})+\ell(R[x,u])+\ell(P[u,a_s])+\ell(Q_1)+\ell(Q[b_i,b_0])$
\end{center}	
\begin{center}
$\ge\ell(P[a_0,a])+\ell(P[a,u])+\ell(Q[u,b_i])+\ell(Q[b_i,b_0])+3~~~~~~~~~$
\end{center}	
\begin{center}
$\ge\frac{p}{2}+\frac{p-4}{2}+3=p+1,~~~~~~~~~~~~~~~~~~~~~~~~~~~~~~~~~~~~~~~~~~~~~~~~~~$
\end{center}	
as desired. 
\end{proof}

We now prove the following claim.
\begin{claim}\label{claim2}
$Q_1$ and $Q_2$ are the only two $(B,A\cup X)$-paths in $G-\{u\}$.
\end{claim}
\begin{proof}
Suppose not. Without loss of generality, we may assume that $\widehat{Q}=b\cdots y$ is a $(B,A)$-path in $G-V(R)$ which is different from $Q_1$, where $b\in B$ and $y\in A$. Recall that $\ell(Q_1)=1$. So by Claim \ref{claim1}, either $y=a_s$ or $b=b_i$. However, this contradicts the fact that $\ell(P[u,a_s])=\ell(Q[u,b_i])+1$. Therefore, Claim 2 is true.
\end{proof}
\begin{center}
\includegraphics[scale=.7]{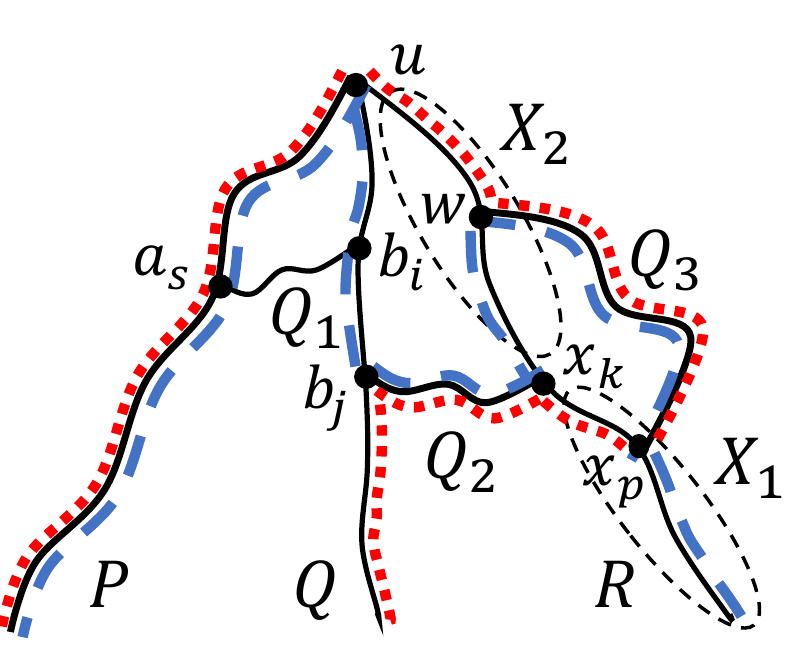}
\end{center}
\begin{center}
Fig. 6. Configurations of $\overline{Q^{'}}$ and $\overline{Q^{''}}$. 
\end{center}

Finally, we get to the heart of the proof for Subcase \ref{subcase2.2}. Note that $X_1\neq\emptyset$, otherwise $PRQ_2$ forms a path of length at least $p+1$. Since $G$ is 3-connected, there exists an $(X_1,A\cup B\cup X_2)$-path $Q_3=x_p\cdots w$ in $G-\{x_k,u\}$, where $x_p\in X_1$. If $w\in X_2$, then let 
\begin{center}
$\overline{Q^{'}}=PQ[u,b_j]Q_2R[x_k,w]Q_3R[x_p,x_0]$
\end{center}
and
\begin{center}
$\overline{Q^{''}}=PR[u,w]Q_3R[x_p,x_k]Q_2Q[b_j,b_0]$.
\end{center}
be two paths of $G$ (see Fig. 6). Thus we have
\begin{center}
$\ell(\overline{Q^{'}})+\ell(\overline{Q^{''}})=\ell(P)+\ell(Q[u,b_j])+\ell(Q_2)+\ell(R[x_k,w])+\ell(Q_3)+\ell(R[x_p,x_0])$
\end{center}
\begin{center}
$~~~~~~~~~~~~~~~~~~~~~~~+\ell(P)+\ell(R[u,w])+\ell(Q_3)+\ell(R[x_p,x_k])+\ell(Q_2)+\ell(Q[b_j,b_0])$
\end{center}
\begin{center}
$=2\ell(P)+\ell(R)+\ell(Q)+2\ell(Q_2)+2\ell(Q_3)~~~~~~~~~~~$
\end{center}
\begin{center}
$\ge 2\cdot\frac{p}{2}+\frac{p}{2}+\frac{p-4}{2}+4=2p+2>2p,~~~~~~~~~~~~~~~~~$
\end{center}
which is a contradiction. Therefore, we can conclude that $w\in A\cup B$. Furthermore, by Claim \ref {claim2}, $w\in A$. In order to avoid conditions (IV) and (V), we see that $w=a_s$. Then by property ($p_4$), we have $\ell(R[u,x_p])-\ell(P[u,a_s])\ge 2$, which contradicts the property $(p_5)$. Hence we complete the proof of Subcase \ref{subcase2.2} and the proof of Case \ref{case2}.

This completes the proof of Theorem \ref{thm2}.\qed

Consider the 4-spoked wheel $W_4$, which is obtained by joining a single vertex and all vertices of a cycle with 4 vertices. Clearly, $W_4$ is 3-connected with any longest path having length 4. As $W_4$ has 5 vertices, each bond $B$ has a corresponding partition $(V_1,V_2)$ of $V(W_4)$ such that $|V_1|\ge|V_2|$, and both $G[V_1]$ and $G[V_2]$ are connected. As $|V_1|\ge3$, $G[V_1]$ has a path of length 2, and this path clearly does not intersect the bond $B$. Therefore, $n\ge6$ is the best possible for Theorem \ref{thm2}.

\textbf{Proof of Theorem \ref{thmx}.} Suppose to the contrary that $G$ does not have such a bond which satisfies the statement. Let $\mathscr{L}$ be the set of all paths of length at least $p-t+1$ in $G$, where $t=\Big\lfloor\sqrt{\frac{k-2}{2}}\Big\rfloor$ if $p$ is even and $t=\Big\lceil\sqrt{\frac{k-2}{2}}\Big\rceil$ if $p$ is odd. Note that $\delta(G)\ge k$ as $G$ is $k$-connected. Assume that $|V(G)|\le2k+1$. Then by Theorem \ref{D}, $G$ contains a Hamilton path when $|V(G)|\le2k$. Furthermore, as $G$ is connected, there also exists a Hamilton path when $|V(G)|=2k+1$. Notice that $p-t+1\ge(|V(G)|-1)-\Big\lceil\sqrt{\frac{|V(G)|-3}{2}}\Big\rceil+1\ge\lceil\frac{|V(G)|}{2}\rceil=\lceil\frac{p+1}{2}\rceil$. It follows that there exists a bond in $G$ meeting all paths of $\mathscr{L}$ by Lemma \ref{lemqq}, which contradicts our assumption. Therefore, we shall assume that $|V(G)|\ge 2k+2$ in the following proof. 

Let $L=x_0x_1\ldots x_p$ be a longest path in $G$. Set $r=\lceil \frac{p}{2}\rceil$. Let $S$ be a vertex set containing $\{x_r,x_{r+1},\ldots,x_p\}$. Let $G_1$ denote a component of $G-S$ containing $\{x_0,x_1,\ldots, x_{r-1}\}$. Define $B^1$ as the set of edges of $G$ with one end in $V(G_1)$ and the other in $S$. Clearly, $B^1$ is a bond of $G$. If each path of $\mathscr{L}$ has a vertex in $V(G_1)$ and another in $S$, then $B^1$ meets all paths of $\mathscr{L}$. Thus there exists a path $L^{'}$ in $\mathscr{L}$ such that $V(L^{'})\subseteq V(G_1)$ or $V(L^{'})\subseteq V(G)-V(G_1)$.

Let $P^{'}=x_i\cdots u$ be a $(V(L),V(L^{'}))$-path in $G$, where $x_i\in V(L)$ and $u\in V(L^{'})$, such that $i$ is the largest index if $V(L^{'})\subseteq V(G_1)$ and the smallest index if $V(L^{'})\subseteq V(G)-V(G_1)$. Here we define $\ell(P^{'})=0$, that is $x_i=u$, when $|V(L)\cap V(L^{'})|\ge1$. Clearly, $u$ divides $L^{'}$ into two subpaths, denoted by $P$ and $Q$. Without loss of generality, we may assume that $\ell(P)\ge\ell(Q)$. Let $R=P^{'}L[x_i,x_j]$ denote a path of $G$ such that $j=p$ if $V(L^{'})\subseteq V(G_1)$ and $j=0$ if $V(L^{'})\subseteq V(G)-V(G_1)$. Now $P,Q$ and $R$ are three paths of $G$ and any two of them intersect only at the vertex $u$. Moreover, $\ell(R)\ge\lceil\frac{p}{2}\rceil$. It is worth noting that $\ell(R)=\lceil\frac{p}{2}\rceil$ if and only if either $x_{r-1}\in V(L^{'})\cap V(L)$ and $p$ is odd or $x_r\in V(L^{'})\cap V(L)$. Note that $p-\ell(R)\ge\ell(P)\ge\ell(Q)$, otherwise $PR$ forms a path of length at least $p+1$. It follows that $\ell(Q)\ge p-t+1-\ell(P)\ge\ell(R)-t+1$, hence $\ell(P)\ge\ell(Q)\ge\ell(R)-t+1$. 
\begin{center}
	\includegraphics[scale=.6]{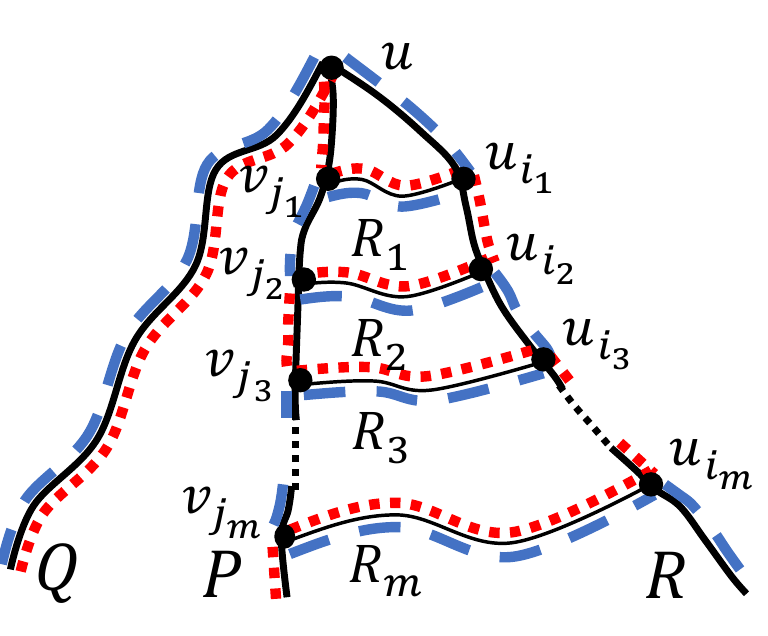}
	\hfil
	\includegraphics[scale=.6]{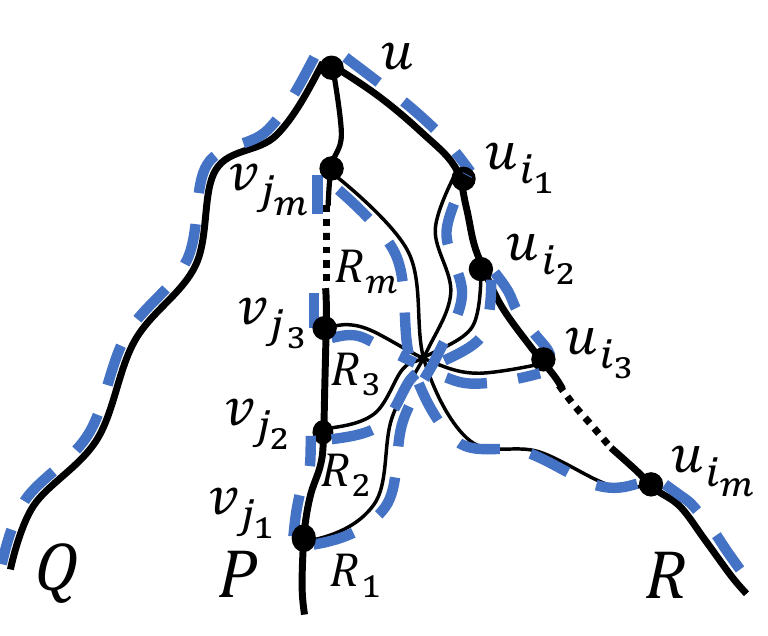}
	\hfil
	\includegraphics[scale=.6]{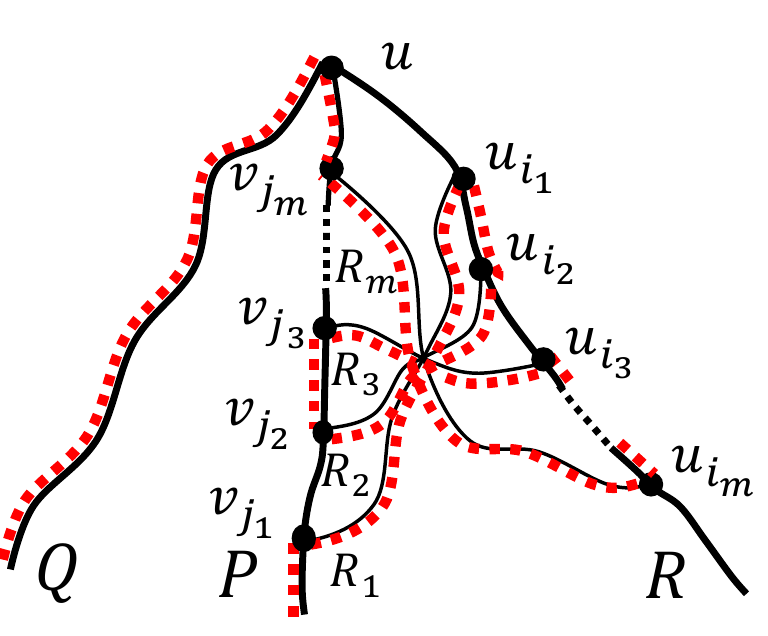}
	\hfil
	($a$)~~~~~~~~~~~~~~~~~~~~~~~~~~~~~~~~~~~~~~($b$)~~~~~~~~~~~~~~~~~~~~~~~~~~~~~~~~~~~~~~($c$)~~~~~
\end{center}
\begin{center}
	Fig. 7. Configurations of two paths, where $m$ is even in $(b)$ and $(c)$.
\end{center}

By Theorem \ref{D}, it is easily seen that $p\ge2k$. Set $A=V(P)-\{u\}, B=V(Q)-\{u\}$ and $X=V(R)-\{u\}$. Then $|A\cup B|\ge p-t+1\ge 2k-\Big\lceil\sqrt{\frac{k-2}{2}}\Big\rceil+1\ge k-1$ and $|X|\ge\lceil\frac{p}{2}\rceil\ge k-1$. As $G$ is $k$-connected, by Lemma \ref{lemfan}, there exist at least $k-1=2((\sqrt{\frac{k-2}{2}}+1)-1)^2+1$ vertex-disjoint $(X,A\cup B)$-paths in $G-\{y\}$. By the pigeonhole principle, $G-V(Q)$ contains at least $s=((\Big\lfloor\sqrt{\frac{k-2}{2}}\Big\rfloor+1)-1)^2+1$ vertex-disjoint $(X,A)$ or $(X,B)$-paths. Without loss of generality, we may assume that there exist at least $s$ vertex-disjoint $(X,A)$-paths. As $\ell(P)\ge\ell(Q)\ge\ell(R)-t+1$, the proof is similar when there exist at least $s$ vertex-disjoint $(X,B)$-paths. Let $w$ and $z$ denote the other endpoints of $P$ and $R$, respectively. Now labeling $u_1,u_2,\ldots,u_s$ as the endpoints of these paths on $X$ in the direction from $u$ to $z$ and labeling the corresponding endpoints of these paths on $A$ by $v_1,v_2,\ldots,v_s$. By Theorem \ref{thmez}, there exists a monotone sequence on $A$, which in turn gives us at least $\Big\lfloor\sqrt{\frac{k-2}{2}}\Big\rfloor+1$ vertex-disjoint $(X,A)$-paths that are either pairwise parallel or pairwise crossing. We arbitrarily choose $m$ of these $(X,A)$-paths, say $R_1=u_{i_1}\cdots v_{j_1},\ldots,R_m=u_{i_m}\cdots v_{j_m}$, where $m$ is even and $m\ge\Big\lfloor\sqrt{\frac{k-2}{2}}\Big\rfloor$.  If these $m$ paths are pairwise parallel, then let 
\begin{center}
$\overline{P_1}=QP[u,v_{j_1}]R_1R[u_{i_1},u_{i_2}]R_2P[v_{j_2},v_{j_3}]R_3\cdots R_mP[v_{j_m},w]$	
\end{center}
and 
\begin{center}
$\overline{Q_1}=QR[u,u_{i_1}]R_1P[v_{j_1},v_{j_2}]R_2R[u_{i_2},u_{i_3}]R_3\cdots R_mR[u_{i_m},z]$	
\end{center}
be two paths in $G$ (see Fig. 7(a)). If these $m$ paths are pairwise crossing, then let 
\begin{center}
$\overline{P_2}=QR[u,u_{i_1}]R_1P[v_{j_1},v_{j_2}]R_2R[u_{i_2},u_{i_3}]R_3\cdots R_mR[u_{i_m},z]$	
\end{center}
and 
\begin{center}
$\overline{Q_2}=P[w,v_{j_1}]R_1R[u_{i_1},u_{i_2}]R_2P[v_{j_2},v_{j_3}]R_3\cdots R_mP[v_{j_m},u]Q$	
\end{center}
be two paths in $G$ (see Fig. 7(b) and 7(c)). For both cases, we have
\begin{center}
$\ell(\overline{P_1})+\ell(\overline{Q_1})=\ell(\overline{P_2})+\ell(\overline{Q_2})=\ell(Q)+\ell(P)+\ell(Q)+\ell(R)+2m~~~~~~~~~~~~~~~~~~~~~~~~$
\end{center}
\begin{center}
$~~~~~~~~~~~~~~~~~~~~~~~~~~~~~\ge\lceil\frac{p}{2}\rceil-t+1+p-t+1+\lceil\frac{p}{2}\rceil+2\cdot\Big\lfloor\sqrt{\frac{k-2}{2}}\Big\rfloor$
\end{center}
\begin{center}
$\ge2p+1>2p,~~~~~~~~~~$
\end{center}
which is a contradiction. 

This completes the proof of Theorem \ref{thmx}.\qed
\section*{Acknowledgements}
The authors declare that there is no conflict of competing interest.

\end{document}